\documentclass{amsart}
\usepackage[utf8]{inputenc}

\title{A note on the Nielsen realization problem for Enriques manifolds}
\date{}
\author[S.\ Billi]{Simone Billi}

\usepackage[english]{babel}
\usepackage[T1]{fontenc}
\usepackage{amssymb,amsmath,amsthm,amsfonts}
\usepackage{thmtools,mathtools}
\usepackage{microtype}
\allowdisplaybreaks
\usepackage{enumitem}
\usepackage{tikz,tikz-cd}
\usepackage[breaklinks,pdfencoding=auto,psdextra]{hyperref}
\usepackage{bm} 
\usepackage[scaled]{beramono}

\usepackage{listings}
\usepackage{caption}
\usepackage{xcolor}
\usepackage{amsmath}
\usepackage{amssymb}
\usepackage{amsfonts}
\usepackage{mathrsfs}
\usepackage{amsthm}
\usepackage{graphicx}
\usepackage{floatflt}
\usepackage{epstopdf}
\usepackage{comment}
\usepackage{tikz-cd}
\usepackage{url}
\usepackage{appendix}

\lstdefinelanguage{Macaulay2}{
comment=[l]{--},
alsoletter={'},
alsoother={_},
}
\newcommand{\hookuparrow}{\mathrel{\rotatebox[origin=c]{0}{$\hookrightarrow$}}}

\DeclareMathOperator{\bL}{\mathbf{L}}
\DeclareMathOperator{\bLambda}{\mathbf{\Lambda}}
\DeclareMathOperator{\Bl}{Bl}

\DeclareMathOperator{\Aut}{Aut} 
\DeclareMathOperator{\og10}{OG10}
\DeclareMathOperator{\Isom}{Isom} 
\DeclareMathOperator{\Ker}{Ker} 
\DeclareMathOperator{\Homology}{H}
\DeclareMathOperator{\Gr}{Gr}
\DeclareMathOperator{\GL}{GL}

\DeclareMathOperator{\NS}{NS}
\DeclareMathOperator{\T}{T}

\DeclareMathOperator{\SDiff}{SDiff}
\DeclareMathOperator{\Diff}{Diff}
\DeclareMathOperator{\Mod}{Mod}
\DeclareMathOperator{\SMod}{SMod}
\DeclareMathOperator{\Mon}{Mon}
\DeclareMathOperator{\Homeo}{Homeo}

\DeclareMathOperator{\E}{E}
\DeclareMathOperator{\B}{B}
\DeclareMathOperator{\M}{M}

\DeclareMathOperator{\id}{id}
\DeclareMathOperator{\SO}{SO}

\DeclareMathOperator{\bO}{\mathrm{O}}
\DeclareMathOperator{\II}{\mathbf{I}}
\DeclareMathOperator{\bU}{\mathbf{U}}
\DeclareMathOperator{\bE}{\mathbf{E}}
\DeclareMathOperator{\bM}{\mathbf{M}}

\newtheorem{thm}{Theorem}[section]
\newtheorem{lem}[thm]{Lemma}
\newtheorem{prop}[thm]{Proposition}
\newtheorem{definition}[thm]{Definition}
\newtheorem{cor}[thm]{Corollary}
\newtheorem{ese}[thm]{Example}
\newtheorem*{problem*}{Problem}
\newtheorem{question}{Question}
\newcommand{\hk}{hyper-K{\"a}hler}

\newcommand{\KTn}{K3\(^{[n]}\)}
\newcommand{\Kumn}{Kum\(_{n}\)}
\newcommand{\KTST}{K3$^{[n]}$ type}

\newtheorem{rmk}{Remark}
\begin{document}
\maketitle

\begin{abstract}
    We give a numerical criterion for the Nielsen realization problem for Enriques manifolds, based on the recent developments on the Birman–Hilden theory for \hk{} manifolds and on Nielsen realization for \hk{} manifolds. We apply the criterion to known examples of Enriques manifolds to get explicit groups that can be realized or not realized, and comment on questions related to the Nielsen realization problem. 
\end{abstract}
\section*{Introduction}
The \textit{Nielsen realization problem} was originally formulated by Nielsen in \cite{nielsen1942abbildungsklassen}, and then affirmatively solved by Kerckhoff in \cite{kerckhoff1983nielsen}. The question is whether any finite group $G$ of mapping classes of a complex curve can be lifted to a group of diffeomorphisms which preserve the metric and the complex structure. Equivalently, the question is if $G$ fixes any point in the Teichm\"uller space.\\
The answer to this problem for K3 surfaces is given by Farb and Looijenga in \cite{farb2021nielsen}, and a generalization to \hk{} manifolds is given by the author in \cite{billi2022note}. A version of the Birman--Hilden for \hk{} manifolds is proved by Raman in \cite{raman2024smooth}, this allows to relate the Nielsen realization problem for these to the Nilsen realization problem for their quotients.

\textit{Enriques manifolds} are complex manifolds whose universal cover is a \hk{} manifold; the classical example of Enriques manifolds is the one of Enriques surfaces, for which the Nielsen realization problem is settled in \cite[Theorem 8.5]{raman2024smooth}. The aim of this note is to discuss, to some extent, the Nielsen realization for higher dimensional Enriques manifolds. The first examples and the first definitions of Enriques manifolds were given in \cite{os,Boissiere.NieperWisskirchen.Sarti:enriques.varieties}, the definition we take in account in this note is the one given in \cite{os} (the definition in \cite{Boissiere.NieperWisskirchen.Sarti:enriques.varieties} allows universal covers to be strictly Calabi-Yau, for example). The fundamental group of an Enriques manifold is cyclic and its order is called the \textit{index} of the manifold. At the time being, all the known examples of Enriques manifolds are obtained as quotients of hyper-Kähler manifolds of \KTn{} type or \Kumn{} type and have index 2, 3 or 4; see \cite{os,Boissiere.NieperWisskirchen.Sarti:enriques.varieties}. It is known that there are no Enriques manifolds arsing as quotients of hyper-Kähler manifolds of OG10 type \cite{billi2025nonexistenceenriquesmanifoldsog10}.


\section*{Setting and results}

\subsection*{Mapping class group and the Birman--Hilden property}
Let \(M\) be a closed oriented differentiable manifold. The \textit{mapping class groups} is given by 
\[\Mod(M):= \pi_0(\Diff^+(M))\]
where \(\Diff^+(M)\) denotes the group of orientation-preserving diffeomorphisms of \(M\). If \(D\leq \Diff^+(M)\) is a subgroup acting properly and freely on \(M\), we let \(\SDiff^+_D(M)\) be the centralizer of \(D\) in \(\Diff^+(M)\). In this situation, following \cite{raman2024smooth}, we define the \textit{symmetric mapping class group} as 
\[\SMod_D(M):= \pi_0(\SDiff^+_D(M)).\]
Let \(Y\) be an Enriques manifold and consider its universal covering map \(p\colon X\to Y\), where \(X\) is \hk{} and \(Y=X/D\) for a group of holomorphic automorphisms \(D\leq \Aut(X)\) acting freely. In this case, \(D\cong\pi_1(Y)\) and it is a cyclic group by \cite[Proposition 2.4]{os}. The covering map \(p\) has the Birman--Hilden property by \cite[Theorem A]{raman2024smooth}, i.e. there is an induced morphism of groups 
\[\widetilde{p}\colon \SMod_D(X)\to \Mod(Y)\] 
which gives an isomorphism \(\SMod_D(X)/D\cong \Mod(Y)\). If \(G\leq \Mod(Y)\) we consider its lift \(\widetilde{G}=\widetilde{p}^{-1}(G) \leq \SMod_D(X)\).

\subsection*{Hyper-Kähler manifolds and a lattice invariant}
Let us recall that a hyper-Kähler manifold \(X\) is a compact Kähler manifold such that \(\pi_1(X)=\{1\}\) and \(\Homology^0(X,\Omega_X^2)= \mathbb{C}\cdot \sigma\) for an everywhere non-degenerate form \(\sigma\). We refer to \cite{debarre2022hyper,gross2012calabi} for an account. 

Hyper-Kähler manifolds that are deformation equivalent to a Hilbert scheme of \(n\geq 2\) points \(S^{[n]}\) on a K3 surface \(S\) are called of \textit{K3\(^{[n]}\) type}. Hyper-Kähler manifolds that are deformation equivalent to \(K_n(A):=s^{-1}(0)\) for \(n\geq 2\), where \(s\colon A^{[n+1]}\to A\) is the Albanese map (taking the sum of points) and \(A\) is an abelian surface, are called of \textit{Kum\(_n\) type}.

Let \(X\) be \hk{}. We can consider the lattice \(\bLambda_X=(\Homology^2(X,\mathbb{Z}),q_X)\) of signature \((3,b_2(X)-3)\), where \(q_X\) is the Beauville--Bogomolov--Fujiki quadratic form (which is a topological invariant by the Fujiki relation), and its group of orientation-preserving isometries \(\bO^+(\bLambda_X)\). 

For the following facts, we refer to the inspirational setting for K3 surfaces \cite{farb2021nielsen}, which was adapted to the case of hyper-Kähler manifolds in \cite{looijenga2021teichmuller,billi2022note}. There is a representation of the mapping class group on cohomology 
\[\rho_X\colon \Mod(X)\to \bO^+(\bLambda_X).\] Since $\Gr^+(3,\bLambda_X\otimes\mathbb{R})$ is the symmetric space of $\bO^+(\bLambda_X\otimes\mathbb{R})$, it is non-positively curved and then a finite group $\widetilde{G}\leq \Mod(X)$ must fix a point $P$. 
This means that $P$ is a $\widetilde{G}$-invariant positive $3$-space, and hence there is a linear representation $\widetilde{G}\rightarrow \SO(P)$ of $\widetilde{G}$. Let $\II_{\widetilde{G}}$ be the sum of all the irreducible $\widetilde{G}$-subrepresentations of $\bLambda_X\otimes\mathbb{R}$ which are isomorphic to any of the ones appearing in $P$, which is indeed not depending on \(P\), and set 
\[ \bL_{\widetilde{G}}:=\II_{\widetilde{G}}^\perp\cap \bLambda_X.\]
We say that a linear form $\delta\in\bLambda_X^\vee$ is \textit{negative} if its kernel has signature $(3,b_2(X)-4)$, that is equivalent to \(\delta^\vee\) having negative square.
If $\mathcal{C}$ is a connected component of the Teichm\"uller space \(\mathcal{T}_{Ein}(X)\) of Einstein metrics on \(X\), we let $\Delta_\mathcal{C}\subset \bLambda_X^\vee$ be the set of indivisible negative forms which are represented by an irreducible rational curve for a \hk{} metric with underlying Einstein metric belonging to $\mathcal{C}$, see \cite[Subsection 5.1]{looijenga2021teichmuller}. 

\subsection*{Results}

As a consequence of \cite[Theorem A and Theorem E]{raman2024smooth}, and \cite[Theorem 1.3]{billi2022note}, we prove the following:
\begin{thm}\label{thm1}
	Let \(Y\) be an Enriques manifold with universal cover \(p\colon X\to Y\) and \(D=\pi_1(Y)\). Let \(G\leq \Mod(Y)\) be a finite subgroup and \(\widetilde{G}=\tilde{p}^{-1}(G)\leq \SMod_D(X)\). \begin{enumerate}
	    \item Then $G$ is realized as a group of isometries for an Einstein metric if and only if $\widetilde{G}$ fixes a connected component $\mathcal{C}$ of $\mathcal{T}_{Ein}(X)$, the lattice $\bL_{\widetilde{G}}^\perp$ contains the trivial representation and $\bL_{\widetilde{G}}\cap\Delta_\mathcal{C}=\emptyset$. 
        \item Moreover, if \(G\) is realized then the Einstein metric can be chosen to be Kähler.
	\end{enumerate}
    
\end{thm}

This gives the possibility to concretely control the realization of groups of mapping classes in \(\Mod(Y)\) via the cohomological action of their lift to \(\SMod(X)\).

We give a remark about \cite[Question 8.9]{raman2024smooth}, concerning Enriques surfaces, in the case of Enriques manifolds. For an Enriques manifold \(X\to Y\), in order to lift a group of mapping classes \(G\leq \Mod(Y)\) to a \(D\)-invariant \(G\)-action on \(X\) the study of the central extension 
\[1\to D\to \SMod(X)\to \Mod(Y)\to 1\]
given by the Birman-Hilden theory is needed. A priori, the group of lifts \(\widetilde{G}\) is just an extension of the group \(G\). The question is then if \(\SMod(X)=\Mod(Y)\times D\), which is certainly implied by \(\Homology^2(\Mod(Y),D)=0\). Our contribution is a partial negative answer for Enriques manifolds, contingent to the existence of Enriques manifolds which are intermediate quotients of hyper-Kähler manifolds. 

\begin{prop}\label{thm:central_extension}
    There exists an Enriques manifold \(Y\) such that the central extension \[1\to D\to \SMod(X)\to \Mod(Y)\to 1\]
    does not admit a section. In particular, in this case we have \(\Homology^2(\Mod(Y),D)\not=0.\)
    

\end{prop}
The question in general remains open, and there is still the possibility of having \(\Homology^2(\Mod(Y),D)=0\) when the Enriques manifold is not an intermediate quotient.
\section*{Outline}

In \autoref{proof_main_thm} we prove \autoref{thm1}. In \autoref{examples} we briefly recall the construction of the known examples of Enriques manifolds constructed in \cite{os,Boissiere.NieperWisskirchen.Sarti:enriques.varieties}, in \autoref{sec:examples_of_realiz} we discuss the realization of some groups of mapping classes for these examples. In \autoref{sec:a_non-trivial_ext} we prove \autoref{thm:central_extension}.

\section*{Acknowledgments}
The author would like to thank Giovanni Mongardi and Stevell Muller for useful comments and for reading a preliminary version of the manuscript. The author is thankful to an anonymous referee for pointing out errors in the previous version.


\section{Proof of the numerical criterion}\label{proof_main_thm}

A combination of the results in \cite{raman2024smooth} and \cite{billi2022note} easily give a proof of \autoref{thm1}. For the proof, we follow the ideas of \cite[Theorem 8.5]{raman2024smooth}.

\begin{proof}[Proof of \autoref{thm1}] 

Let us prove (1).
Suppose that \(G\) is realized as a group of isometries for an Einstein metric \(h\) on \(Y\). Clearly the lift \(\widetilde{G}\) preserves a lifted Einstein metric \(\widetilde{h}\) on \(X\), and hence a positive \(3\)-plane \(P_{\widetilde{h}}\in \Gr^+(3,\bLambda_X\otimes \mathbb{R})\). This gives a representation \(\widetilde{G}\to \SO(P_{\widetilde{h}})\cong \SO(3)\) and, since \(\widetilde{G}\) is in the centralizer of \(D\), the elements of \(D\) and elements of \(\widetilde{G}\) commute and in particular their images must commute in \(\SO(3)\). From the fact that the dimension of a maximal torus of \(\SO(3,\mathbb{R})\) is \(1\), the elements of \(\widetilde{G}\) and \(D\) must have a common \(1\)-eigenspace. This gives a line \(l\subset P_{\widetilde{h}}\) which defines a \(\widetilde{G}\)-invariant complex structure. In conclusion, we see that \(\widetilde{G}\) lifts to a group of isometries for a Kähler-Einstein metric on \(X\), where Kähler is basically given by \cite[Theorem 8.6]{raman2024smooth}. From the second item of \cite[Theorem 1.3]{billi2022note} it follows that \(\widetilde{G}\) fixes a component \(\mathcal{C}\) for which \(\bL_{\widetilde{G}}\cap \Delta_\mathcal{C}=\emptyset\) together with the fact that \(\bL_{\widetilde{G}}^\perp\) contains the trivial representation. The viceversa of the statement also follows.

We now prove (2). Suppose that the lift $\widetilde{G}$ fixes a connected component $\mathcal{C}$ of $\mathcal{T}_{Ein}(X)$, that the lattice $\bL_{\widetilde{G}}^\perp$ contains the trivial representation and that $\bL_{\widetilde{G}}\cap\Delta_\mathcal{C}=\emptyset$. Then by \cite[Theorem 1.3]{billi2022note} we have that \(\widetilde{G}\) lifts to a group of isometries for a Kähler-Einstein metric \(\widetilde{h}\) on \(X\). Arguing as in the proof of \cite[Theorem 8.5]{raman2024smooth} we have a central extension 
\[1\to D\to\widetilde{G}\to G\to 1,\]
from which we deduce that any Kähler-Einstein metric preserved by \(\widetilde{G}\) is also \(D\)-invariant and hence descends to a \(G\)-invariant Kähler-Einstein metric \(h\) on \(Y\).

\end{proof}
\section{The known examples}\label{examples}
In this section we recall the known examples of Enriques manifolds, which were introduced in \cite{Boissiere.NieperWisskirchen.Sarti:enriques.varieties,os}.
\subsection{Enriques manifolds from K3 surfaces}\label{sec:examples_K3}
Let \(S\) be a complex K3 surface with a fixed-point free holomorphic (indeed algebraic) involution \(\iota \in \Aut(S)\), then the surface \(E=S/\langle \iota \rangle\) is an Enriques surface. Consider the induced involution \(\iota^{[n]}\) on the Hilbert scheme of points \(X=S^{[n]}\) with \(n\geq 1\) is odd, then he was shown in \cite[Proposition 4.1]{os} that the quotient \(Y=S^{[n]}/\langle\iota^{[n]}\rangle\) is an Enriques manifold of complex dimension \(2n\).

More generally, let \(S\) be as above with Picard rank \(\rho(S)=10\) and fix a vector \(v=(r,l,\chi -r)\in \Homology^{ev}(S,\mathbb{Z})=\Homology^{0}(S,\mathbb{Z})\oplus\Homology^{2}(S,\mathbb{Z})\oplus \Homology^{4}(S,\mathbb{Z})\) with \(v^2:=l^2-2r(r-\chi)\geq 0\) and \(\chi \in\mathbb{Z}\) odd. For a very general polarization \(H\in \NS_\mathbb{R}(S)\), the moduli space \(X=M_H(S,v)\) of \(H\)-semistable sheaves on \(S\), with Mukai vector \(v\), admits a fixed-point free involution whose quotient \(X\to Y\) is an Enriques manifold of dimension \(v^2+2\) by \cite[Theorem 5.3]{os}.

In the above cases we have \(D\cong \mathbb{Z}/2\mathbb{Z}\). Moreover, we have the following lattice isometries
\begin{align*}
\bLambda_X &:=\Homology^2(X,\mathbb{Z})\cong \bU^{\oplus 3}\oplus \bE_8(-1)^{\oplus 2}\oplus [-2(n-1)]    \\
\bLambda_Y &:=\Homology^2(Y,\mathbb{Z})_{tf}\cong\bU\oplus \bE_8(-1)\oplus [-(n-1)]
\end{align*} where \(\Homology^2(Y,\mathbb{Z})_{tf}\) denotes the torsion-free part of the cohomology. If \(p\colon X\to Y\) denotes the covering map, there are induced maps
\begin{align*}
    p^* & \colon \Homology^2(Y,\mathbb{Z})\to \Homology^2(X,\mathbb{Z})\\
    p_!&\colon \Homology^2(X,\mathbb{Z})\to \Homology^2(Y,\mathbb{Z})
\end{align*}
where the composition \(p_!\circ p^*\) is the multiplication by \(2\), see \cite[Example 4.4]{raman2024smooth}. 
\subsection{Enriques manifolds from abelian surfaces}
Let \(S\) be a \textit{bielliptic surface}, i.e. an algebraic complex surface which is not abelian but admits a finite étale cover by an abelian surface \(A\). The covering \(A\to S\) turns out to be cyclic and we denote its order by \(d\). 
It was shown in \cite[Theorem 6.4]{os} and \cite[Proposition 4.1]{Boissiere.NieperWisskirchen.Sarti:enriques.varieties} that if \(d|n+1\) then the covering automorphism induces an automorphism of \(K_n(A)\) and in some cases the action is free, leading to an Enriques manifold \(X\to Y\) where \(X=K_n(A)\) and the universal covering map is the quotient map. The only suitable values are \(d=2,3,4\) and in this case we have \(D\cong \mathbb{Z}/d\mathbb{Z}\).
Moreover, we have lattice isometries
\begin{align*}
\Homology^2(X,\mathbb{Z})\cong \bU^{\oplus 3}\oplus [-2(n+1)]    \\
\Homology^2(Y,\mathbb{Z})_{tf}\cong \bU\oplus [-2(n+1)/d]
\end{align*} and again, if \(p\colon X\to Y\) denotes the covering map, there are induced maps
\begin{align*}
    p^* & \colon \Homology^2(Y,\mathbb{Z})\to \Homology^2(X,\mathbb{Z})\\
    p_!&\colon \Homology^2(X,\mathbb{Z})\to \Homology^2(Y,\mathbb{Z})
\end{align*}
where the composition \(p_!\circ p^*\) is the multiplication by \(d\), see \cite[Example 4.6]{raman2024smooth}. 


\section{Some examples of realizations and non-realizations}\label{sec:examples_of_realiz}

If \(\bM\) is a lattice, then we use the notation \(\Gamma_2^+(\bM)\leq\bO^+(\bM)\) for the kernel of the reduction modulo \(2\) map \(\bO^+(\bM,\mathbb{Z})\to 
\bO(\bM,\mathbb{Z}/2\mathbb{Z})\). The group \(\Gamma_2^+(\bM)\) is usually called the group of \textit{2-congruences}. In our case we set \(\bM:=\bU\oplus \bE_8(-1)\).

Clearly, some groups of mapping classes can be realized via automorphisms.

\begin{ese}[Some realizations] \label{ex:realizations}
The following groups are realizable. 
    \begin{itemize}
        \item If \(X=S^{[n]}\) with \(n\) odd and \(X\to Y\) an Enriques manifold of index \(d=2\) as in \cite{os}, then any finite \(G\leq \Gamma_2^+(\bM)\) is realizable. In fact, by \cite[Theorem 3.3]{barth1983automorphisms} \(G\) lifts to a group of automorphisms of \(S\) commuting with the Enriques involution and this induces a group of automorphisms of \(X\) which descens to a group of automorphisms \(G\) of \(Y\).


 \item Let \(A=E\times E\) where \(E\cong \mathbb{C}/(\mathbb{Z}\oplus \xi_d \mathbb{Z})\) with \(\xi_d\) a primitive \(d\)-root of \(1\), and \(d=2,3,4\). Let \(\psi(x,y)=(\xi_d x+1/d,y+1/d)\) for \((x,y)\in A\), then \(\psi\) induces a fixed-point free automorphism on \(X=K_n(A)\) when \(d| n+1\), which gives an Enriques manifold \(X\to Y=X/\langle\psi\rangle\) (as in \cite[Proposition 4.1]{Boissiere.NieperWisskirchen.Sarti:enriques.varieties} and \cite[Theorem 6.4]{os}). Let \(y_0\in E[d]\) be a \(d\)-torsion point, then translation by \((0,y_0)\) on \(A\) commutes with \(\psi\) and preserves \(K_n(A)\), hence it induces an automorphism of \(Y\) which gives a cyclic group of mapping classes \(G\leq\Mod(Y)\) if \(y_0\not= 1/d\). By construction \(G\) is realizable as a group of isometries for a Kähler-Einstein metric and \(G\) acts trivially on the lattice \(\Homology^2(Y,\mathbb{Z})_{tf}\).
    \end{itemize}
\end{ese}

In the following we also exhibit groups of mapping classes that do not admit a lift to a group of isometries for an Einstein metric, applying the criterion of \autoref{thm1}.

In case \(X\) is of K3\(^{[n]}\) then \(\bM\) is the unimodular summand of \(\bLambda_Y\). Notice that there is a morphism of groups
\[\gamma\colon \bO^+_D(\bLambda_X)\to \bO^+(\bLambda_Y)\]
where \(\bO^+_D(\bLambda_X)\) denotes the centralizer of the induced action of \(D\) in \(\bO^+(\bLambda_X)\).
By Nikulin theory, we have that any isometry in the 2-congruence subgroup \(\Gamma_2^+(\bM)\) of \(\bO^+(\bM)\) induces a unique isometry of \(\bLambda_X\) via the embedding \(\bM(2)\subset \bLambda_X\), which restricts to the identity on the sublattice \(\bM(2)^\perp\subset \bLambda_X\), see \cite[Lemma 1.3]{barth1983automorphisms}. By construction these isometries commute with the action of \(D\), and the above discussion gives an embedding
\[\Gamma_2^+(\bM)\hookrightarrow \bO^+_D(\bLambda_X).\]
We have the following commutative diagram 
    \begin{equation*}\label{inclusions}
\begin{tikzcd}
\SMod_D(X) \arrow[r] \arrow[r, "\widetilde{p}"] \arrow[d, "\rho_X"'] & \Mod(Y) \arrow[d, "\rho_Y"] \\
\bO^+_D(\bLambda_X) \arrow[r, "\gamma"]                              & \bO^+(\bLambda_Y)      
\end{tikzcd}
\end{equation*}
where the lower horizontal map is an isomorphism when restricted to \(\Gamma_2^+(\bM)\).

\begin{ese}[Some non-realizations]\label{ex:non_realizations} Let \(X\to Y\) be an Enriques manifold as in \cite{os} with \(X\) of K3\(^{[n]}\) type with \(n\) odd. Then by \cite[Corollary 1.3]{billi2022note} the map \(\Mod(X)_\mathcal{C}\to\bO^+(\bLambda_X)\) has a section over its image \(\Mon^2(X)\), the monodromy group, where \(\Mod(X)_\mathcal{C}\) is the stabilizer of a component \(\mathcal{C}\subset\mathcal{T}_{Ein}\).
    \begin{itemize}
        \item Consider the isometry of the lattice \(\bLambda_Y=\Homology^2(Y,\mathbb{Z})_{tf}\) acting by \((-1)\) on the generator of \([-(n-1)]\) and by the identity on its complement, this involution is of monodromy \cite{markman2011survey}. The isometry extends to an isometry of \(\bLambda_X=\Homology^2(X,\mathbb{Z})\), the section of the representation map gives an order two subgroup \(\widetilde{G}\leq \SMod(X)\) and this descends to an order two subgroup \(G\leq \Mod(Y)\). The group \(G\) is not realizable as isometries of an Einstein metric since the generator of \(\bL_{\widetilde{G}}=[-2(n-1)]\) belongs to \(\Delta_{\mathcal{C}}\).
            \item Similarly, the reflection in \(\bLambda_Y\) along a vector of square \(-2\) is monodromy and it is a 2-congruence, hence as before it gives a group of mapping classes \(G\leq \Mod(Y)\) which is not realizable. 
        
    \end{itemize}

\end{ese}
    In the case of surfaces, the second kind of non-realizable reflections can be represented by a Dehn twist, a diffeomorphism which is not topologically isotopic to any finite order diffeomorphism but whose square is smoothly isotopic to the identity \cite[Theorem D]{raman2024smooth}.
            We expect that such \(G\) is also represented by a diffeomorphism whose square is smoothly isotopic to the identity. The first kind of non-realizable reflections has no counterpart in the surface case.

\color{black}

\section{A non-trivial Birman-Hilden central extension}\label{sec:a_non-trivial_ext}



We now observe that for an Enriques manifold constructed as a quotient of a generalized Kummer manifold, with fundamental group \(D\cong \mathbb{Z}/d\mathbb{Z}\) for \(d\) not a prime number, there exists another Enriques manifold \(Y\) such that \(\Homology^2(\Mod(Y),D)\not=0\).

\begin{proof}[Proof of {\autoref{thm:central_extension}}]
Assume \(Z\) is an Enriques manifold with \(\pi_1(Z)\cong \mathbb{Z}/d\mathbb{Z}\) where \(d\) is not prime. This means that there exists a hyper-Kähler manifold \(X\) and \(\varphi \in \Aut(X)\) such that \(Z=X/\langle \varphi \rangle\). In this case also \(Y=X/\langle\varphi^k\rangle\) is an Enriques manifold, where we choose some \(1<k< d\) dividing \(d\). We have \(D=\langle\varphi^k\rangle\) and we can see \(\varphi\) as an automorphism of \(Y\) of order \(m=\frac{d}{k}\). Assume moreover that the natural map \(\langle\phi\rangle\to \SMod(X)\) is injective. Then the central exact sequence 
    \[1\to D\to \SMod(X) \to \Mod(Y)\to 1\] when restricted to the group generated by \(\varphi\) gives the exact sequence 
    \[1\to \langle \varphi^k \rangle\to  \langle \varphi\rangle\to \langle \varphi \rangle/\langle \varphi^k \rangle\to 1,\] which is equivalent to the central extension 
    \[0\to \mathbb{Z}/k\mathbb{Z}\to \mathbb{Z}/d\mathbb{Z}\to \mathbb{Z}/m\mathbb{Z}\to 0,\] that admits no sections. In particular, \(\Homology^2(\Mod(Y),D)\not=0\).

    The situation for example happens in the index \(4\) examples of \cite[Proposition 4.1]{Boissiere.NieperWisskirchen.Sarti:enriques.varieties} or \cite[Theorem 6.4]{os}, where \(X=K_n(A)\) admits a free quotient \(Z=X/\langle \varphi \rangle\) with \(\varphi\) of order \(4| n+1\). In this case \(k=m=2\) and \(d=4\), the map \(\langle\phi\rangle\to \SMod(X)\) is injective since \(\Aut(X)\to \GL(\Homology^*(X,\mathbb{Z}))\) is faithful by \cite[Theorem 1.3]{OGUISO_2020}
\end{proof}

\bibliographystyle{alpha}

\bibliography{references}
\end{document}